\documentclass{article}

\RequirePackage[OT1]{fontenc}
\RequirePackage[
amsthm,amsmath,natbib]{imsart}
\usepackage{xcolor}
\usepackage[normalem]{ulem}

\usepackage{graphicx}

\usepackage{latexsym,amsmath}
\usepackage{amsmath,amsthm,amscd}
\usepackage{amsfonts}
\usepackage{calrsfs}
\usepackage[psamsfonts]{amssymb}
\usepackage{pstricks,pst-plot}
\usepackage{enumerate}

\usepackage{url}

\usepackage{tocvsec2}

\usepackage{epstopdf}

\usepackage{wrapfig}
\usepackage{float}

\usepackage[active]{srcltx}

\theoremstyle{plain} 
\newtheorem{theorem}{Theorem}[section]
\newtheorem{corollary}[theorem]{Corollary}

\newtheorem{lemma}[theorem]{Lemma}
\newtheorem{proposition}[theorem]{Proposition}

\theoremstyle{definition} 

\theoremstyle{definition} 

\newtheorem*{ex*}{Example}
\theoremstyle{remark} 

\theoremstyle{remark} 
\newtheorem{remark}[theorem]{Remark}
\newtheorem*{remark*}{Remark}
\numberwithin{equation}{section}

\newcommand{\mathsym}[1]{{}}
\newcommand{\unicode}[1]{{}}

\newcommand{\iincludegraphics}[1]{{}}

\renewcommand{\le}{\leqslant}
\renewcommand{\ge}{\geqslant}

\newcommand{\ii}[1]{\operatorname{I}\left\{#1\right\}}

\newcommand{\eqs}{\overset{\operatorname{sign}}=}

\newcommand{\dd}{\,{\mathrm d}}

\newcommand{\supp}{\operatorname{supp}}
\newcommand{\card}{\operatorname{card}}

\DeclareMathOperator*{\argmax}{argmax}

\newcommand{\num}{\operatorname{num}}
\newcommand{\den}{\operatorname{den}}

\newcommand{\R}{\mathbb{R}}

\newcommand{\E}{\operatorname{\mathsf{E}}}
\renewcommand{\P}{\operatorname{\mathsf{P}}}
\newcommand{\MM}{\operatorname{\mathsf{M}}}

\newcommand{\de}{\delta}
\newcommand{\vp}{\varepsilon}

\newcommand{\la}{\lambda}
\newcommand{\si}{\sigma}
\newcommand{\Si}{\Sigma}
\renewcommand{\th}{\theta}
\newcommand{\Th}{\Theta}

\newcommand{\td}{\tilde d}

\newcommand{\tX}{\tilde X}
\newcommand{\tY}{\tilde Y}

\newcommand{\tvp}{\tilde\vp}

\newcommand{\A}{\mathcal{A}}
\newcommand{\M}{\mathcal{M}}
\renewcommand{\SS}{\mathcal{S}}
\newcommand{\X}{\mathcal{X}}
\newcommand{\PP}{\mathcal{P}}

\newcommand{\XX}{\mathbf{X}}
\newcommand{\xx}{\mathbf{x}}

\newcommand{\pd}[2]{\frac{\partial#1}{\partial#2}}

\newcommand{\sint}{\textstyle{\int}} 

\begin{document}

\begin{frontmatter}

\title{Exact bounds on the truncated-tilted mean, with applications
}
\runtitle{Bounds on the truncated-tilted mean
}
\date{\today}

\begin{aug}
\author{\fnms{Iosif} \snm{Pinelis}\ead[label=e1]{ipinelis@mtu.edu}}
\runauthor{Iosif Pinelis}

\affiliation{Michigan Technological University}

\address{Department of Mathematical Sciences\\
Michigan Technological University\\
Houghton, Michigan 49931, USA\\
E-mail: \printead[ipinelis@mtu.edu]{e1}
}
\end{aug}

\begin{abstract}
Exact upper bounds on the Winsorised-tilted mean, $
\frac{\E Xe^{h(X\wedge w)}}{\E e^{h(X\wedge w)}}$, of a random variable $X$ in terms of its first two moments are given. 
Such results are needed in work on nonuniform Berry--Esseen-type bounds for general nonlinear statistics. 
As another application,  
optimal upper bounds on the Bayes posterior mean are provided. 
Certain monotonicity properties of the tilted mean are also presented. 
\end{abstract}

 

\end{frontmatter}

\settocdepth{chapter}

\tableofcontents 

\settocdepth{subsubsection}

\section{Introduction}\label{intro}

Cram\'er's tilt transform of a random variable (r.v.) $X$ is a r.v.\ $X_h$ such that 
\begin{equation*}
	\E f(X_h)=\frac{\E f(X)e^{h\,X}}{\E e^{h\,X}}
\end{equation*}
for all nonnegative Borel functions $f$, where $h$ is a real parameter. This transform is an important tool in the theory of large deviation probabilities $\P(X>x)$, where $x>0$ is a large number; then the appropriate value of the parameter $h$ is positive.  
Unfortunately, if the right tail of the distribution of $X$ decreases slower than exponentially, then $\E e^{h\,X}=\infty$ for all $h>0$ and thus the tilt transform is not applicable. 
The usual recourse then is to replace $X$ in the exponent by its truncated counterpart, say $X\ii{X\le w}$ or $X\wedge w$, where $w$ is a real number. 
As shown in \cite{winzor,nonlinear}, of the two mentioned kinds of truncation, it is the so-called Winsorization, $X\wedge w$, of the r.v.\ $X$ that is more useful in the applications considered there. 

In particular, in \cite{nonlinear} one needs a good upper bound on the mean 
\begin{equation}\label{eq:E_h}
	\E_{h,w}X:=\frac{\E Xe^{h(X\wedge w)}}{\E e^{h(X\wedge w)}}. 
\end{equation}
of the Winsorised-tilted distribution of $X$. 
Note that $\E_{h,w}X$ is well defined and finite for any $h\in(0,\infty)$, any $w\in\R$, and any r.v.\ $X$ such that $\E(0\vee X)<\infty$. 

In \cite{winzor}, exact upper bounds on the denominator $\E e^{h(X\wedge w)}$ of the ratio in \eqref{eq:E_h} were provided, which effected 
a significant improvement on previously obtained upper bounds on $\E_{h,w}X$ (also, \cite{winzor} contained applications to pricing of certain financial derivatives). 
However, before the present study, the numerator of the ratio in \eqref{eq:E_h} was bounded separately from the denominator, which entailed a serious loss in accuracy. 
In this paper, exact upper bounds on $\E_{h,w}X$ will be provided, in terms of the first two moments of the r.v.\ $X$. 
In fact, without loss of generality one may assume that $\E X=0$, since 
\begin{equation}\label{eq:shift}
	\E_{h,w}(X+m)=m+\E_{h,w-m}X
\end{equation}
for any $m\in\R$ \big(and, again, any $h\in(0,\infty)$ and $w\in\R$). 

We shall show (Theorem~\ref{th:sup=sup}) that the maximum of $\E_{h,w}X$ is attained at a r.v.\ $X$ with  taking one just two values. 
This allows for further analysis leading to a rather easily computable expression for the maximum of $\E_{h,w}X$, as well as simple and explicit, but at the same time optimal, upper bounds on this maximum; these latter results are provided in Theorem~\ref{th:1}. 
We shall also present various monotonicity properties of the maximum \big(part (II) of Theorem~\ref{th:sup=sup}, and Proposition~\ref{prop:mono}\big), and also demonstrate uniqueness of the maximizer \big(part (I) of Theorem~\ref{th:1}\big). 
In addition, we shall apply some of these results to obtain optimal upper bounds on the Bayes posterior mean. 

\section{Summary and discussion}\label{results}

Take any $h$ and $\si$ 
in $(0,\infty)$ and any $w\in\R$. 
It will sometimes be more convenient to state the results, not in terms of r.v.'s $X$, but in terms of the corresponding probability distributions $P$, so that we shall use the notation  
\begin{equation*}
	\MM_{h,w}P:=\frac{\int_\R xe^{h(x\wedge w)}P(\dd x)}{\int_\R e^{h(x\wedge w)}P(\dd x)}
\end{equation*}
instead of $\E_{h,w}X$ --- with $\MM$ standing for ``the (Winsorised-tilted) mean''. 

Let $\PP$ denote the set of all probability distributions (that is, Borel probability measures) on $\R$. 
Let then 
\begin{equation}\label{eq:PP}
\begin{aligned}
	\PP_\si&:=\{P\in\PP\colon\sint_\R\, x\,P(\dd x)=0,\;\sint_\R\, x^2\,P(\dd x)=\si^2\}, \\ 
	\PP_{\le\si}&:=\{P\in\PP\colon\sint_\R\, x\,P(\dd x)=0,\;\sint_\R\, x^2\,P(\dd x)\in(0,\si^2]\}, \\ 
	\PP^{(2)}_{w,\si}&:=\{P_{u,v}\colon -\infty<-u<w\le v<\infty,\;uv=\si^2\}, \\ 
	\PP^{(2)}_{w,\le\si}&:=\{P_{u,v}\colon -\infty<-u<w\le v<\infty,\;uv\in(0,\si^2]\},  
\end{aligned}	
\end{equation}
where, for any positive real numbers $u$ and $v$, the symbol $P_{u,v}$ stands for the unique zero-mean probability distribution on the two-point set $\{-u,v\}$.  
Note here that (i) the conditions $-\infty<-u<w\le v<\infty$ and $uv>0$ 
imply that $u>0$ and $v>0$; and (ii) $\sint_\R\, x^2\,P_{u,v}(\dd x)=uv$; so, 
$\PP^{(2)}_{w,\si}\subset\PP_\si$ and $\PP^{(2)}_{w,\le\si}\subset\PP_{\le\si}$.   

Let $\X_\si$ denote the class of all r.v.'s whose probability distributions belong to the set $\PP_\si$; similarly define the classes $\X_{\le\si}$, $\X^{(2)}_{w,\si}$, and $\X^{(2)}_{w,\le\si}$. 

Define now the corresponding suprema: 
\begin{equation}\label{eq:SS}
\begin{alignedat}{2}
	\SS_{h,w,\si}:=&\sup\{\MM_{h,w}P\colon P\in\PP_\si\}&&=\sup\{\E_{h,w}X\colon X\in\X_\si\}, \\
	\SS_{h,w,\le\si}:=&\sup\{\MM_{h,w}P\colon P\in\PP_{\le\si}\}&&=
					\sup\{\E_{h,w}X\colon X\in\X_{\le\si}\}, \\
	\SS^{(2)}_{h,w,\si}:=&\sup\{\MM_{h,w}P\colon P\in\PP^{(2)}_{w,\si}\}&&=
					\sup\{\E_{h,w}X\colon X\in\X^{(2)}_{w,\si}\}, \\
	\SS^{(2)}_{h,w,\le\si}:=&\sup\{\MM_{h,w}P\colon P\in\PP^{(2)}_{w,\le\si}\}&&=
	\sup\{\E_{h,w}X\colon X\in\X^{(2)}_{w,\le\si}\}.   
\end{alignedat}	
\end{equation}
Consider also the attainment sets for these suprema: 
\begin{equation}\label{eq:A}
\begin{aligned}
	\A_{h,w,\si}:=&\{P\in\PP_\si\colon\MM_{h,w}P=\SS_{h,w,\si}\}, \\
	\A_{h,w,\le\si}:=&\{P\in\PP_{\le\si}\colon\MM_{h,w}P=\SS_{h,w,\le\si}\}, \\
	\A^{(2)}_{h,w,\si}:=&\{P\in\PP^{(2)}_{w,\si}\colon\MM_{h,w}P=\SS^{(2)}_{h,w,\si}\}, \\
	\A^{(2)}_{h,w,\le\si}:=&\{P\in\PP^{(2)}_{w,\le\si}\colon\MM_{h,w}P=\SS^{(2)}_{h,w,\le\si}\}.   
\end{aligned}	
\end{equation}

We shall say, interchangeably, that some or all of the four suprema in \eqref{eq:SS} (and the related suprema in \eqref{eq:SSk} below) are attained at a r.v.\ $X$ or at a probability distribution $P$, assuming that $P$ is the distribution of $X$.

\begin{theorem}\label{th:sup=sup}
The following statements hold:  
\begin{enumerate}[(I)]
	\item the four suprema in \eqref{eq:SS} are all the same: 
\begin{equation}\label{eq:sup=sup}
\SS_{h,w,\le\si}=\SS_{h,w,\si}=\SS^{(2)}_{h,w,\si}=\SS^{(2)}_{h,w,\le\si};   
\end{equation}	
	\item 
each of the four suprema in \eqref{eq:sup=sup} is (strictly) increasing in $\si\in(0,\infty)$; 	
	\item 
each of these suprema is attained, and 
\begin{equation}\label{eq:A=A}
\A_{h,w,\le\si}=\A_{h,w,\si}=\A^{(2)}_{h,w,\si}=\A^{(2)}_{h,w,\le\si}.    
\end{equation}
\end{enumerate}
\end{theorem}

The proofs will be given in Section~\ref{proofs}. 

Let us now show how to compute effectively the four equal suprema in Theorem~\ref{th:sup=sup};  
in particular, we shall see that each of the four attainment sets in \eqref{eq:A=A} is a singleton one. 
We shall also provide simple (and, in a sense, optimal) upper bounds on these 
suprema; such bounds are what was needed in \cite{nonlinear}. 

\begin{remark}\label{rem:scaling} 
The shift-transformation formula \eqref{eq:shift} allowed us to reduce the consideration to zero-mean distributions. 
One can also do rescaling, to reduce the set of all possible values of the Winsorization level $w$ from $\R$ to $\{-1,0,1\}$.  
Indeed, 
observe that 
\begin{equation*}
	\E_{h,w}X=|w|\,\E_{h|w|,\,w/|w|}\frac X{|w|}
\end{equation*}
for any real $w\ne0$ (and any $X\in\X_{\le\si}$), which  
implies  
$$\SS_{h,w,\si}=|w|\,\SS_{h|w|,\,w/|w|,\,\si/|w|}$$
and the similar formulas for the other three suprema in \eqref{eq:sup=sup}. 
So, without loss of generality let us assume that $w\in\{-1,0,1\}$, which will allow us  
to simplify the writing.  
\end{remark}  

To state Theorem~\ref{th:1} below, more notation is needed. 
For any $\vp\in(0,\infty)$, let 
\begin{equation}\label{eq:uu}
u_*(h,\vp):=\vp^2\,	\frac{e^{(1+\vp)h}-1-\vp h}{ 1+\vp h-e^{-(1+\vp)h}}. 
\end{equation}

Let $\eqs$ mean ``equals in sign to''.

The following proposition allows one to define terms used in the statement of Theorem~\ref{th:1}. 

\begin{proposition}\label{prop:def}
\ 
\begin{enumerate}[(i)]
	\item 
There is a unique root $\si_h\in(0,\infty)$ of the equation 
\begin{equation*}
u_*(h,\si^2_h)=\si^2_h.   	
\end{equation*}
\item
If $\si>\si_h$, then there is a unique root $\tvp_{h,\si}\in(0,\si^2)$ of the equation 
\begin{equation}\label{eq:vp_h,si}
u_*(h,\tvp_{h,\si})=\si^2;    	
\end{equation}
moreover, $u_*(h,\vp)-\si^2\eqs\vp-\tvp_{h,\si}$ for all $\vp\in(0,\si^2)$. 
	\item 
For each $w\in\{-1,0\}$, there is a unique root $\vp_{h,w,\si}\in(|w|,\infty)$ of the equation 
\begin{equation}\label{eq:vp_-1}
\begin{aligned}
r_{w,1}(\vp_{h,w,\si})&=  0,\ \text{where}\\  
	r_{w,1}(\vp)&:= e^{(\varepsilon+w) h} (1+\varepsilon h) \left(\varepsilon ^2+\sigma ^2\right)
	-\varepsilon ^2e^{2 (\varepsilon+w) h}-\sigma ^2. 
\end{aligned}   
\end{equation}
\end{enumerate}
\end{proposition}

Also, let us recall that $P_{u,v}$ stands for the unique zero-mean probability distribution on the two-point set $\{-u,v\}$. 
 
\begin{theorem}\label{th:1}
Take any $w\in\{-1,0,1\}$. Then the following statements hold. 
\begin{enumerate}[(I)]
	\item Each of the four attainment sets in \eqref{eq:A=A} coincides with the singleton set 
\begin{equation}\label{eq:A=}
	\A_{h,w,\si}=\{P_{\vp_{h,w,\si},\,\si^2/\vp_{h,w,\si}}\}, 
\end{equation}
where $\vp_{h,w,\si}$ is defined by \eqref{eq:vp_-1} for $w\in\{-1,0\}$, and  
\begin{align}\label{eq:vp=}
\vp_{h,1,\si}:=&
	\left\{
	\begin{alignedat}{2}
	&\si^2 &\text{ if } \si&\le\si_h, \\
	&\tvp_{h,\si} &\text{ if } \si&>\si_h, 
	\end{alignedat}
	\right.
\end{align}
with $\tvp_{h,\si}$ defined by equation \eqref{eq:vp_h,si}. 
	\item 
Moreover, 
\begin{equation}\label{eq:SS=}
	\SS_{h,w,\si}=\MM_{h,w}P_{\vp_{h,w,\si},\,\si^2/\vp_{h,w,\si}}
	<K_w(h)\si^2, 
\end{equation}
where 	
\begin{equation*}
	K_w(h):=\left\{
	\begin{alignedat}{2}
	&e^h-1 &\text{ if } w&=1, \\
		&\frac h{-L_{-1}(-e^{h w-1})} &\text{ if } w&\in\{-1,0\}, 
	\end{alignedat}
	\right.
\end{equation*} 
$L_{-1}$ is a branch of the Lambert product-log function such that for each $z\in[-e^{-1},0)$ the value $L_{-1}(z)$ is the only root $u\in(-\infty,-1]$ of the equation $u e^u=z$ (see e.g.\ \cite{knuth96} concerning properties of the Lambert function).  
One may observe that (i) $K_0(h)=h$ and (ii) $K_{-1}(h)\sim h$ as $h\downarrow0$ and $K_{-1}(h)\to1$ as $h\to\infty$.  
	\item
The constant factor $K_w(h)$ in \eqref{eq:SS=} is the best possible. 
\end{enumerate}
\end{theorem} 

Of course, in view of \eqref{eq:sup=sup}, the supremum $\SS_{h,w,\si}$ can be replaced in \eqref{eq:SS=} by any of the other three suprema. 

\begin{remark}\label{rem:unique}
By Remark~\ref{rem:scaling} and \eqref{eq:A=}, for any given $w\in\R$ the four attainment sets in \eqref{eq:A=A} coincide with the same singleton set; that is, each of the four suprema in \eqref{eq:sup=sup} is attained at the same unique maximizer.  
\end{remark}

Let us now propose a complement to part (II) of Theorem~\ref{th:sup=sup}; in fact, this proposition is a corollary to Theorems~\ref{th:sup=sup} and \ref{th:1}.  
To state it, let $\supp\mu$ denote, as usual, the support (set) of any Borel measure $\mu$ on $\R$; 
so, $\supp\mu$ is the complement of the union of all open sets $O\subseteq\R$ with $\mu(O)=0$; equivalently, $\supp\mu$ is the set of all points $x\in\R$ such that $\mu(O)>0$ for all open sets $O\subseteq\R$ containing the point $x$. 
For any r.v.\ $X$, let $\supp X$ denote the support of the measure that is the probability distribution of $X$; also, let 
\begin{equation*}
	i_X:=\inf\supp X\quad \text{and}\quad s_X:=\sup\supp X;  
\end{equation*}
note that one may have $i_X=-\infty$ and/or $s_X=\infty$. 

\begin{proposition}\label{prop:mono}
The following statements hold:   
\begin{enumerate}[(I)]
	\item 
for any r.v.\ $X$, $\E_{h,w}X$ is nondecreasing in $h\in(0,\infty)$ and in $w\in\R$; 
	\item 
for any r.v.\ $X$ with $i_X<s_X$ and any $w\in(i_X,\infty)$, 
$\E_{h,w}X$ is increasing in $h\in(0,\infty)$;
	\item 
for any r.v.\ $X$ and any $h\in(0,\infty)$, 
$\E_{h,w}X$ is increasing in $w\in[i_X,s_X]\cap\R$; 
	\item 
each of the four equal suprema in \eqref{eq:sup=sup} is increasing in $h\in(0,\infty)$ and in $w\in\R$. 
\end{enumerate}
\end{proposition}



The case of a positive Winsorization level $w$ is the one most important in applications. 
In accordance with Remark~\ref{rem:scaling}, this case is represented in Theorem~\ref{th:1} by $w=1$. 
Although the corresponding upper bound $(e^h-1)\si^2$ on $\SS_{h,1,\si}$ is very simple and the constant factor $K_1(h)=e^h-1$ in it is optimal, the relative error of this bound is small only if $h$ or $\si$ is small, as illustrated in (the right panel of) Figure~\ref{fig:}, 
showing the graphs of the ratios of $\SS_{h,w,\si}$ to $K_w(h)\si^2$. 
However, it is small values of $\si$ that are of particular interest in the application of Theorem~\ref{th:1} in \cite{nonlinear}. 

\begin{figure}[H]
\includegraphics[width=0.90\textwidth]{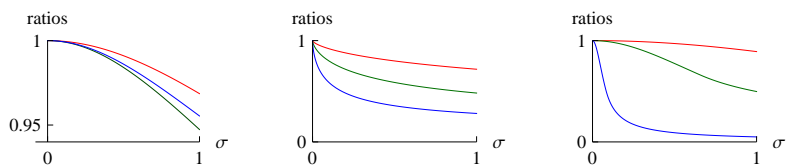}
	\caption{Graphs of the ratios of $\SS_{h,w,\si}$ to $K_w(h)\si^2$ for $w=-1$ (left panel), $w=0$ (middle panel), $w=1$ (right panel), 
	$h=\frac15,1,5$ (red, green, blue --- respectively), and $\si\in(0,1]$. }
	\label{fig:}
\end{figure}

\noindent For $w=-1$, the relative errors are seen to be rather small even for $h$ as large as $5$ and $\si$ as large as $1$. Also, the relative errors appear to be monotonic in $\si$, but not in $h$ or in $w$.


It is obvious that the upper bounds on $\E_{h,w}X$ will hold if the factor $X$ of $e^{h(X\wedge w)}$ in the numerator of the ratio in \eqref{eq:E_h} is replaced by any r.v.\ that is no greater than $X$. Thus, by Theorem~\ref{th:1} and Remark~\ref{rem:scaling}, one has 

\begin{corollary}\label{cor:Y}
Take any $h$, $w$, and $\si$ in $(0,\infty)$. 
Take then any r.v.\ $X\in\X_{\le\si}$ and let $Y$ be any r.v.\ such that $Y\le X$ with probability $1$. 
Then 
\begin{equation*}
	\frac{\E Ye^{h(X\wedge w)}}{\E e^{h(X\wedge w)}}\le\SS_{h,w,\si}<\frac{e^{hw}-1}w\,\si^2. 
\end{equation*}
In particular, one can take here $Y=X\wedge w$ or $Y=X\ii{X\le w}$. 
\end{corollary} 

\begin{remark}\label{rem:Y}
Suppose that $h$, $w$, $\si$, and $X$ are as in Corollary~\ref{cor:Y} and, in addition, $\P(X\le w)=1$.  Then, by Corollary~\ref{cor:Y},   
\begin{equation}\label{eq:X}
	\frac{\E Xe^{hX}}{\E e^{hX}}\le\SS_{h,w,\si}<\frac{e^{hw}-1}w\,\si^2. 
\end{equation}
Moreover, according to Theorem~\ref{th:1}, the bound $\SS_{h,w,\si}$ on $\frac{\E Xe^{hX}}{\E e^{hX}}$ 
is still exact \big(even under the additional condition $\P(X\le w)=1$\big) --- provided that $\si\le w\si_h$; cf.\ \eqref{eq:vp=}. 
As pointed out before, in the applications that motivated this study, the values of $\si$ are typically small and thus will satisfy the condition $\si\le w\si_h$. 
Also, the proof of Theorem~\ref{th:1} shows \big(see especially the paragraph containing formula \eqref{eq:rho=rho=}\big) that the factor $\frac{e^{hw}-1}w$ in the second upper bound in \eqref{eq:X} will still be optimal, even under the additional restriction $\si\le w\si_h$ --- because the supremum of $\rho_{h,1}(\vp)$ in $\vp\in E_1=(0,\infty)$ is ``attained'' in the limit as $\vp\downarrow0$, and, in turn, $\rho_{h,1}(\vp)$ is a supremum in $\si$, which is attained at $\si=\si_1(\vp)=\sqrt\vp$, which latter goes to $0$ as $\vp\downarrow0$.  
\end{remark}

At this point, one is ready to present 

\subsection{Application: Optimal prior bounds on the Bayes posterior mean for exponential families}\label{prior} 
Consider a so-called exponential family $\big(p(\cdot|\th)\big)_{\th\in\Th}$ of probability densities on some measurable ``data'' space $\XX$ with respect to some measure on $\XX$; that is, 
\begin{equation*}
	p(\xx|\th)=e^{\th T(\xx)}c(\th)q(\xx)
\end{equation*}
for some positive Borel-measurable function $c$ on some nonempty Borel-measur\-able ``parameter space'' $\Th\subseteq\R$, some nonnegative measurable functions $T$ and $q$ on $\XX$, all $\th\in\Th$, and all $\xx\in\XX$; 
one also needs to require that $p(\xx|\th)$ be measurable in $(\xx,\th)$. 
Thus, $\th$ is what is usually referred as the natural parameter. 
For instance, for the family of the Poisson distributions with parameter $\la\in(0,\infty)$, the natural parameter is $\th=\ln\la\in\R=\Th$ and $c(\th)=e^{-\la}=e^{-e^\th}$. 

Suppose that $\th_{\max}:=\sup\Th<\infty$.  
Further, let $\pi$ be a Borel measure on $\Th$, which will play the role of a so-called prior distribution. 
Note that $\pi$ does not have to be a probability distribution. In fact, it will be convenient here to normalize $\pi$ and/or the function $c$ so that $\int_\Th c(\th)\,\pi(\!\dd\th)=1$. 
Let us exclude the trivial case when $\supp\pi=\{\th_{\max}\}$. 
Then, clearly, $m:=\int_\Th \th\,c(\th)\,\pi(\!\dd\th)<\th_{\max}$. 
Finally, suppose that the variance of the (renormalized by the factor $c$) prior distribution is known to be bounded:  
$\int_\Th(\th-m)^2\,c(\th)\,\pi(\!\dd\th)\le\si^2$ for some $\si\in(0,\infty)$; such an assumption appears especially reasonable in empirical Bayes settings, when accumulated prior knowledge may greatly reduce the uncertainty about the value of $\th$. 

Consider now the posterior mean 
\begin{equation*}
	\hat\th(\xx):=\frac{\int_\Th \th\,p(\xx|\th)\,\pi(\!\dd\th)}{\int_\Th \,p(\xx|\th)\,\pi(\!\dd\th)}
	=\frac{\int_\Th \th\,e^{\th t}\,c(\th)\,\pi(\!\dd\th)}{\int_\Th \th\,e^{\th t}\,c(\th)\,\pi(\!\dd\th)} 
\end{equation*}
given some observable ``data'' $\xx\in\XX$ \big(with $q(\xx)\ne0$\big), where $t:=T(\xx)$. 
Then, by formula \eqref{eq:shift} and Remarks~\ref{rem:Y} and \ref{rem:scaling}, 
one has 
\begin{equation}\label{eq:Bayes}
	\hat\th(\xx)-m\le\SS_{t,\th_{\max}-m,\si}<\frac{e^{(\th_{\max}-m)t}-1}{\th_{\max}-m}\,\si^2,   
\end{equation}
again for $t=T(\xx)$ and any $\xx\in\XX$ \big(with $q(\xx)\ne0$\big). 
If $\si^2$ is small enough, then the upper bounds in \eqref{eq:Bayes} on $\hat\th(\xx)-m$ may be much smaller than the trivial bound $\th_{\max}-m$. 
Also, by Remark~\ref{rem:Y}, the bound $\SS_{t,\th_{\max}-m,\si}$ on $\hat\th(\xx)-m$ is exact, and the factor $\frac{e^{(\th_{\max}-m)t}-1}{\th_{\max}-m}$ in the second upper bound in \eqref{eq:Bayes} is optimal. 
Recall that the main concern with the Bayesian approach is uncertainty about the choice of the prior distribution. So, the bounds in \eqref{eq:Bayes} may be of help, as they rely only on the first two moments and an upper bound of the support of such a distribution.

\section{Proofs}\label{proofs}


An approach one could try to use to establish an exact upper bound on the ratio $\E_{h,w}X$, defined in \eqref{eq:E_h}, is to fix --- besides the first two moments of $X$ --- a value of the denominator, $\E e^{h(X\wedge w)}$, and then maximize the numerator, $\E Xe^{h(X\wedge w)}$, subject to these three (affine) restrictions on the measure that is the distribution of $X$. In fact, here one has one more, less explicit restriction on this measure, which can be written as $\E1=1$, of course meaning that the measure is a probability one. Then one can use some of well-known results such as those in \cite{hoeff-extr,karr} to reduce the optimization problem to the case when $\supp X$ consists of at most four points, corresponding to the four restrictions on the measure. Then the problem will be reduced to calculus with $12$ variables ($8$  variables for the four support points of the measure and the four corresponding masses; one variable for the previously fixed value of $\E e^{h(X\wedge w)}$; and also the parameters $h$, $w$, and $\si$). 
Of these $12$ variables, four can be eliminated using the four restrictions on the measure, and one of the parameters $h,w,\si$ can be eliminated by rescaling. 
Yet, this would leave $7$ variables and a highly nonlinear function to maximize, with a number of restrictions on the variables. 
Such a problem appears too difficult, in terms of the amount of required calculations, especially symbolic ones. 

Here this difficulty is overcome mainly by a thorough exploitation of the duality principle, the idea of which goes back to Chebyshev; see e.g.\ \cite{krein59,karlin-studden,krein-nudelman,kemper-dual,pin-games-transl,pin98}. A general expression of this duality is the so-called minimax duality, which goes back to von Neumann \cite{von_Neumann}; see also e.g.\ \cite{ky_fan,kneser,sion,pin-games-transl}; in particular, a necessary and sufficient condition for minimax duality for concave-convex functions 
was given in \cite{pin-games-transl}. 
However, more convenient in a number of problems in probability and statistics turn out to be sufficient conditions for duality given by Kemperman \cite{kemper-dual}, which will be used in the present paper as well. 
Another significant idea in the proof of the basic Theorem~\ref{th:sup=sup} in this paper is a reduction of the maximization of the ratio $\E_{h,w}X=\frac{\E Xe^{h(X\wedge w)}}{\E e^{h(X\wedge w)}}$ of two affine functions (of the distribution of $X$) to the maximization of a linear combination $\E Xe^{h(X\wedge w)}-k\E e^{h(X\wedge w)}=\E(X-k)e^{h(X\wedge w)}$ of these two functions, with an appropriately chosen value of the constant $k$. 
As the result, we show that the maximum of $\E_{h,w}X$ is attained at a r.v.\ $X$ with $\supp X$ consisting just of two (rather than four) points. 
This allows for further analysis, to be presented in the proof of Theorem~\ref{th:1}, resulting in a rather easily computable expression for the maximum of $\E_{h,w}X$, as well as simple and explicit, but at the same time optimal, upper bounds on this maximum. 

In accordance to some of the above discussion, Theorem~\ref{th:sup=sup} is obtained as a rather easy corollary of 
Theorem~\ref{th:k} 
below. 
To state the latter, 
let us 
introduce, for all $k\in\R$,  
\begin{equation}\label{eq:SSk}
\begin{aligned}
	\SS_{k;h,w,\si}&:=\sup\{\sint_\R\,(x-k)e^{h(x\wedge w)}P(\dd x)\colon P\in\PP_\si\}
	\\ &=\sup\{\E(X-k)e^{h(X\wedge w)}\colon X\in\X_\si\}, \\
	\SS_{k;h,w,\le\si}&:=\sup\{\sint_\R\,(x-k)e^{h(x\wedge w)}P(\dd x)\colon P\in\PP_{\le\si}\}
	\\ &=\sup\{\E(X-k)e^{h(X\wedge w)}\colon X\in\X_{\le\si}\}, \\
	\SS^{(2)}_{k;h,w,\si}&:=\sup\{\sint_\R\,(x-k)e^{h(x\wedge w)}P(\dd x)\colon P\in\PP^{(2)}_{w,\si}\}
	\\ &=\sup\{\E(X-k)e^{h(X\wedge w)}\colon X\in\X^{(2)}_{w,\si}\}, \\
	\SS^{(2)}_{k;h,w,\le\si}&:=\sup\{\sint_\R\,(x-k)e^{h(x\wedge w)}P(\dd x)\colon P\in\PP^{(2)}_{w,\le\si}\}
	\\ &=\sup\{\E(X-k)e^{h(X\wedge w)}\colon X\in\X^{(2)}_{w,\le\si}\}. 
\end{aligned}	
\end{equation}
Similarly to \eqref{eq:A}, define now 
the attainment sets 
$\A_{k;h,w,\si}$, $\A_{k;h,w,\le\si}$, $\A^{(2)}_{k;h,w,\si}$, and $\A^{(2)}_{k;h,w,\le\si}$ (as subsets of $\PP$) 
for the corresponding suprema in \eqref{eq:SSk}.  

\begin{theorem}\label{th:k}
For any $k\in\R$   
\begin{enumerate}[(I)]
	\item the four suprema in \eqref{eq:SSk} are all the same: 
\begin{equation}\label{eq:(k)sup=sup}
\SS_{k;h,w,\le\si}=\SS_{k;h,w,\si}=\SS^{(2)}_{k;h,w,\si}=\SS^{(2)}_{k;h,w,\le\si};    
\end{equation}
	\item 
each of the four suprema in \eqref{eq:(k)sup=sup} is (strictly) increasing in $\si\in(0,\infty)$; 	
	\item 
each of these suprema is attained, and 
\begin{equation}\label{eq:(k)A=A}
\A_{k;h,w,\le\si}=\A_{k;h,w,\si}=\A^{(2)}_{k;h,w,\si}=\A^{(2)}_{k;h,w,\le\si}.    
\end{equation}
\end{enumerate} 
\end{theorem}

\begin{proof}[Proof of Theorem~\ref{th:sup=sup} (modulo Theorem~\ref{th:k})]
Note first that all the four suprema in \eqref{eq:sup=sup} are finite, because 
$\X^{(2)}_{w,\si}\ne\emptyset$ and, for 
any r.v.\ $X\in\X_{\le\si}$, one has 
$|\E Xe^{h(X\wedge w)}|\le\E|X|e^{hw}\le\si e^{h|w|}$, 
$\E e^{h(X\wedge w)}\ge e^{h\E(X\wedge w)}
\ge e^{-h(\E|X|+|w|)}\ge e^{-h(\si+|w|)}$, and hence 
$|\E_{h,w}X|\le \si e^{h(\si+2|w|)}$. 
%
If now $k$ is chosen to coincide with $\SS^{(2)}_{h,w,\si}$, then 
\begin{equation}\label{eq:X-k}
\E(X-k)e^{h(X\wedge w)}
=(\E_{h,w}X-k)\E e^{h(X\wedge w)}\le0 
\end{equation} 
for all $X\in\X^{(2)}_{w,\si}$, so that $\SS^{(2)}_{k;h,w,\si}
\le0$; in fact, $\SS^{(2)}_{k;h,w,\si}=0$, since the factor $\E e^{h(X\wedge w)}$ stays between $0$ and $e^{hw}$, and hence is bounded. 
So, by Theorem~\ref{th:k}, $\SS_{k;h,w,\si}
=0$. 
Therefore, 
\begin{equation}\label{eq:le k ...}
\E Xe^{h(X\wedge w)}\le k\E e^{h(X\wedge w)}	
\end{equation}
for all $X\in\X_\si$, which implies that $\SS_{h,w,\si}\le k=\SS^{(2)}_{h,w,\si}$. The reverse inequality, $\SS_{h,w,\si}\ge\SS^{(2)}_{h,w,\si}$, 
is trivial. 
So, the second equality in \eqref{eq:sup=sup} is verified. 
Moreover, by Theorem~\ref{th:k}, the supremum $\SS^{(2)}_{k;h,w,\si}=0$ is attained at some r.v.\ $X\in\X^{(2)}_{w,\si}$, for which inequality \eqref{eq:le k ...} must then turn into the equality $\E Xe^{h(X\wedge w)}=k\E e^{h(X\wedge w)}$, which is equivalent to $\E_{h,w}X=k$. 
So, the suprema $\SS^{(2)}_{h,w,\si}$ and $\SS_{h,w,\si}$ are attained. 

Let us now verify the monotonicity of $\SS_{h,w,\si}(=\SS^{(2)}_{h,w,\si})$ in $\si\in(0,\infty)$ \big(which will also yield the first and third equalities in \eqref{eq:sup=sup}, as well as the attainment of the suprema $\SS_{h,w,\le\si}$ and $\SS^{(2)}_{h,w,\le\si}$\big). 
Here the reasoning is similar to that in the previous paragraph, 
again with $k=\SS^{(2)}_{h,w,\si}$, which, as was shown, equas to 
$\SS_{h,w,\si}$. Then relations \eqref{eq:X-k} hold for all $X\in\X_\si$, so that $\SS_{k;h,w,\si}=0$. 
Take now any $\si_1\in(0,\si)$. 
Then, by part (II) of Theorem~\ref{th:k}, $\SS_{k;h,w,\si_1}<\SS_{k;h,w,\si}=0$. 
Again using the equality in \eqref{eq:X-k} (with its two sides interchanged), one has $\E_{h,w}X<k=\SS_{h,w,\si}$ for all $X\in\X_{\si_1}$, which implies $\SS_{h,w,\si_1}<\SS_{h,w,\si}$ --- since, by what has been already proved, the supremum $\SS_{h,w,\si_1}$ is attained at some $X\in\X_{\si_1}$. 

It remains to observe that \eqref{eq:A=A} follows from \eqref{eq:(k)A=A}. 
Indeed, for any $P_*\in\PP$, one has $P_*\in\A_{h,w,\si}$ if and only if $P_*\in\A_{k;h,w,\si}$ for $k=\SS_{h,w,\si}$ and at that $\SS_{k;h,w,\si}=0$; and the 
same is true for each of the pairs of the subscript/superscript attributes  $({}_{h,w,\le\si},\;{}_{k;h,w,\le\si})$, $({}^{(2)}_{h,w,\si},\;{}^{(2)}_{k;h,w,\si})$, $({}^{(2)}_{h,w,\le\si},\;{}^{(2)}_{k;h,w,\le\si})$ in place of the pair $({}_{h,w,\si},\;{}_{k;h,w,\si})$. 
\end{proof}

\begin{proof}[Proof of Theorem~\ref{th:k}] 
Introduce more notation. 
First, let $S$ stand for the set $[-\infty,\infty]$, equipped with the natural topology and the corresponding Borel sigma-algebra; then $S$ is compact. 
Next, take indeed any $k\in\R$ and define the real-valued functions $a_1,a_2,a_3,b$ on $S$ by the formulas 
\begin{equation*}
	a_j(s):=\frac{s^{j-1}}{1+s^2}\quad\text{and}\quad b(s):=\frac{(s-k)e^{h(s\wedge w)}}{1+s^2}
\end{equation*}
for all $j\in\{1,2,3\}$ and $s\in\R$ (assuming the convention $0^0:=1$), with the values of these functions on the set $\{-\infty,\infty\}$ defined by continuity, so that $a_j(\pm\infty)=b(\pm\infty)=0$ for $j\in\{1,2\}$ and $a_3(\pm\infty)=1$.  

Further, let $\M$ stand for the set of all (nonnegative) Borel measures on $S$. 
Introduce now the sets 
\begin{align}
	A:=&\big\{\big(\sint a_1\dd\mu,\sint a_2\dd\mu,\sint a_2\dd\mu\big)\in\R^3\colon\mu\in\M,\ \notag \\  \ &\hspace*{5cm}\sint(|a_1|+|a_2|+|a_3|)\dd\mu<\infty\big\}, \notag \\
	\M_\si:=&\big\{\mu\in\M\colon \big(\sint a_1\dd\mu,\sint a_2\dd\mu,\sint a_2\dd\mu\big)=
	(1,0,\si^2)\big\}, \notag \\ 
	C:=&\big\{(c_1,c_2,c_3)\in\R^3\colon 
	c_1a_1(s)+c_2a_2(s)+c_3a_3(s)\ge b(s)\text{ for all }s\in S\big\},  \label{eq:C}
\end{align}

\vspace*{-10pt}
\noindent where  
the integrals are over $S$. 
Also, let 
\begin{align}
B(c_1,c_2,c_3):=&\{s\in S\colon c_1a_1(s)+c_2a_2(s)+c_3a_3(s)=b(s)\}; \label{eq:B} 
\end{align}
for $(c_1,c_2,c_3)\in C$, the set $B(c_1,c_2,c_3)$ is sometimes referred to as the contact set --- compare \eqref{eq:C} and \eqref{eq:B}. 

Observe the following.  
\begin{itemize}
	\item 
For all $\mu\in \M_\si$, one has $\sint|b|\dd\mu<\infty$, since $|b|=O(a_1+a_3)$. 
For the same reason, $C\ne\emptyset$. Moreover, the strict inequality $ca_1+ca_3>b$ holds (on $S$) for	some large enough $c\in(0,\infty)$ (depending on $h$ and $k$). 
	\item 
The point $(1,0,\si^2)$ lies in the interior of the set $A\subseteq\R^3$. This follows because the condition $\si\in(0,\infty)$ implies that (i) there is a measure $\mu\in \M_\si$ with $\card\,\supp\mu=3$ and (ii) the restrictions of the functions $a_1,a_2,a_3$ to the three-point set $\supp\mu$ are linearly independent.   
\end{itemize}
\rule{0pt}{0pt} 
(As usual, $\card$ denotes the cardinality.) 
Therefore, by Theorems~3 and 4 in \cite{kemper-dual} \big(see also comments in the penultimate paragraph of \cite[Section~3]{kemper-dual}\big),  
there exist $(c^\circ_1,c^\circ_2,c^\circ_3)\in C$ and $\mu^\circ\in \M_\si$ such that 
\begin{equation}\label{eq:saddle}
	  \supp\mu^\circ\subseteq B(c^\circ_1,c^\circ_2,c^\circ_3)\quad\text{and}\quad
	  \sup\{\sint b\dd\mu\colon\mu\in \M_\si\}=\sint b\dd\mu^\circ. 	   
\end{equation}
Moreover, the conditions $\si\in(0,\infty)$ and $\mu^\circ\in \M_\si$ imply that $\card\supp\mu^\circ\ge2$. 
So, by Lemma~\ref{lem:contact} below, 
\begin{equation}\label{eq:!}
\supp\mu^\circ=\{-u,v\}	
\end{equation}
for some real numbers $u$ and $v$ such that $-u<w\le v$; 
in particular, $\supp\mu^\circ\subseteq\R$. 

Next, observe that the formula
\begin{equation}\label{eq:mu,nu}
	P(\dd s)=\frac{\mu(\dd s)}{1+s^2}
\end{equation}
defines a one-to-one correspondence between the set $\{\mu\in \M_\si\colon\supp\mu\subseteq\R\}$ and the set $\PP_\si$ of zero-mean probability measures on $\R$, defined in \eqref{eq:PP} 
\big(the formal meaning of \eqref{eq:mu,nu} is of course that the Radon--Nikodym derivative of $P$ relative to $\mu$ is the function $s\mapsto\frac1{1+s^2}$\big). 
Moreover, for any so-related measures $\mu$ and $P$, one has 
$\int_S b\dd\mu=
\int_\R(x-k)e^{h(x\wedge w)}P(\dd x)$. 
Now, in view of \eqref{eq:saddle} and \eqref{eq:SSk},  
\begin{equation}\label{eq:multline}
\begin{aligned}
	\SS_{k;h,w,\si}=&\sup\{\sint_\R\,(x-k)e^{h(x\wedge w)}P(\dd x)\colon P\in\PP_\si\} \\
	=&\sup\{\sint_S\, b\dd\mu\colon\mu\in \M_\si,\; \supp\mu\subseteq\R\} \\
	\le&\sup\{\sint_S\, b\dd\mu\colon\mu\in \M_\si\} \\
	=&\sint_S\, b\dd\mu^\circ \\
	=&\sint_\R\,(x-k)e^{h(x\wedge w)}P^\circ(\dd x)
	\le\SS^{(2)}_{k;h,w,\si}  
	\le\SS_{k;h,w,\si},   
\end{aligned}	
\end{equation}
where $P^\circ$ is 
the measure corresponding to $\mu^\circ$ by means of \eqref{eq:mu,nu}, so that, by \eqref{eq:!},  $P^\circ\in\PP^{(2)}_{w,\si}\subseteq\PP_\si$.  
Thus, both suprema $\SS_{k;h,w,\si}$ and $\SS^{(2)}_{k;h,w,\si}$ are attained at $P=P^\circ$ and are equal to each other, so that the second equality in \eqref{eq:(k)sup=sup} holds. 

Next, let us prove the monotonicity of $\SS_{k;h,w,\si}(=\SS^{(2)}_{k;h,w,\si})$ in $\si\in(0,\infty)$, which will also yield the first and third equalities in \eqref{eq:(k)sup=sup}, as well as the attainment of the suprema $\SS_{k;h,w,\le\si}$ and $\SS^{(2)}_{k;h,w,\le\si}$. 
Take any $\si_1\in(0,\si)$ and then take any $P\in\PP_{\si_1}$. 

Let now the measure $\mu\in \M_{\si_1}$ with $\supp\mu\subseteq\R$ be defined 
by the formula $\mu(\dd s)=(1+s^2)P(\dd s)$, in accordance with the correspondence \eqref{eq:mu,nu}. 
It follows that 
\begin{multline}
	\sint_\R\,(x-k)e^{h(x\wedge w)}P(\dd x)
	=\sint b\dd\mu
	\le\int(c^\circ_1a_1+c^\circ_2a_2+c^\circ_3a_3)\dd\mu
	=c^\circ_1+c^\circ_3\si_1^2 \\ 
	<c^\circ_1+c^\circ_3\si^2
	=\sint b\dd\mu^\circ
	=\SS_{k;h,w,\si};  \label{eq:<}
\end{multline}
here, 
\begin{itemize}
	\item 
the first equality follows by the definition of $\mu$; 
	\item 
the first inequality, by the condition $(c^\circ_1,c^\circ_2,c^\circ_3)\in C$; 
	\item 
the second equality, because $\mu\in \M_{\si_1}$; 
	\item 
the second inequality --- because, by 
Lemma~\ref{lem:contact} below, the condition $(c^\circ_1,c^\circ_2,c^\circ_3)\in C$ implies $c^\circ_3>0$, while $\si_1^2<\si^2$; 
	\item 
the third equality, because $\mu^\circ\in \M_\si$; 		
	\item 
the last equality, by \eqref{eq:multline}.  		
\end{itemize}
Now one can see that $\SS_{k;h,w,\si_1}\le\SS_{k;h,w,\si}$; moreover, this latter inequality is strict, since the last inequality in \eqref{eq:<} is strict and, by what was already proved, the supremum $\SS_{k;h,w,\si_1}$ is attained.  
This concludes the proof of parts (I) and (II) of Theorem~\ref{th:k} --- modulo Lemma~\ref{lem:contact}. 

Next, let us prove \eqref{eq:(k)A=A}. First here, note that the obvious relation $\PP_\si\subseteq\PP_{\le\si}$, together with \eqref{eq:(k)sup=sup}, implies $\A_{k:h,w,\si}\le\A_{k:h,w,\le\si}$, and the reverse inequality follows by the already checked \emph{strict} monotonicity of $\SS_{k;h,w,\si}$ in $\si$. 
So, one has the first equality in \eqref{eq:(k)A=A}, and the third equality there is verified similarly. 

To obtain a contradiction, suppose now that the second equality in \eqref{eq:(k)A=A} is false, so that there exists some $P_*\in\A_{k;h,w,\si}\setminus\A^{(2)}_{k;h,w,\si}$. 
Then, again by Lemma~\ref{lem:contact}, the set $S^\circ:=\supp\mu_*\setminus B(c^\circ_1,c^\circ_2,c^\circ_3)$ is nonempty, where 
$\mu_*\in\M_\si$ is the measure corresponding to $P_*$ in accordance with the correspondence $\eqref{eq:mu,nu}$ and $(c^\circ_1,c^\circ_2,c^\circ_3)\in C$ is as before;  
moreover, the condition $\supp\mu_*\not\subseteq B(c^\circ_1,c^\circ_2,c^\circ_3)$ implies $\mu_*(S^\circ)>0$, because the functions $a_1$, $a_2$, $a_3$, and $b$ are continuous and hence the set $B(c^\circ_1,c^\circ_2,c^\circ_3)$ is closed. 
So and because $c^\circ_1a_1+c^\circ_2a_2+c^\circ_3a_3\ge b$ on $S$ and $c^\circ_1a_1+c^\circ_2a_2+c^\circ_3a_3>b$ on $S^\circ$, it follows that 
\begin{equation}\label{eq:>0}
\sint(c^\circ_1a_1+c^\circ_2a_2+c^\circ_3a_3-b)\dd\mu_*
\ge\sint_{\!S^\circ}\,(c^\circ_1a_1+c^\circ_2a_2+c^\circ_3a_3-b)\dd\mu_*>0.	
\end{equation}
Now one can write
\begin{multline*}
\SS_{k;h,w,\si}=\sint b\dd\mu_* < \sint(c^\circ_1a_1+c^\circ_2a_2+c^\circ_3a_3)\dd\mu_*
=\sint(c^\circ_1a_1+c^\circ_2a_2+c^\circ_3a_3)\dd\mu^\circ \\ 
=\sint b\dd\mu^\circ
=\SS_{k;h,w,\si}, 	
\end{multline*}
which is indeed a contradiction; 
the first equality here follows by the condition $P_*\in\A_{k;h,w,\si}$, the first inequality, by \eqref{eq:>0}; the second equality, because $\mu_*$ and $\mu^\circ$ are in $\M_\si$; the penultimate equality, since $\supp\mu^\circ\subseteq B(c^\circ_1,c^\circ_2,c^\circ_3)$; and the last equality, by \eqref{eq:multline}. 

Thus, to complete the proof of Theorem~\ref{th:k}, it remains to verify

\begin{lemma}\label{lem:contact}
For any $(c_1,c_2,c_3)\in C$, one has $c_3>0$ and, moreover, 
the set $B(c_1,c_2,c_3)$ is empty, a singleton, or of the form 
$\{-u,v\}$ for some real numbers $u$ and $v$ such that $-u<w\le v$; so, in all cases 
$\card B(c_1,c_2,c_3)\le2$. 
\end{lemma}

\begin{proof}[Proof of Lemma~\ref{lem:contact}]
Take any $(c_1,c_2,c_3)\in C$, so that, by \eqref{eq:C}, 
\begin{equation}\label{eq:contact}
\begin{aligned}
	d(s):=
	c_1a_1(s)+c_2a_2(s)+c_3a_3(s)-b(s)
	\ge0
\end{aligned}	
\end{equation}
for all $s\in S=[-\infty,\infty]$ and, by \eqref{eq:B},  
\begin{equation}\label{eq:B=}
	B(c_1,c_2,c_3)=\{s\in S\colon d(s)=0\}. 
\end{equation}
Since $d(\infty)=c_3$, inequality \eqref{eq:contact} implies that $c_3\ge0$. 
Moreover, if $c_3=0$, then $s\,d(s)\to c_2-e^h$ as $s\to\infty$ and $s\,d(s)\to c_2$ as $s\to-\infty$, whence $e^h\le c_2\le0$, which is a contradiction. 
We conclude that indeed  
\begin{equation}\label{eq:c3>0}
	c_3>0. 
\end{equation}
In turn, this implies that $d(s)=c_3\ne0$ for $s\in\{-\infty,\infty\}$. 
Therefore, by \eqref{eq:B=}, 
$
	B(c_1,c_2,c_3)\cap\{-\infty,\infty\}=\emptyset 
$, 
so that  
\begin{gather}
	B(c_1,c_2,c_3)=\{s\in\R\colon\td(s)=0\}, \notag 
	\quad\text{where}\\
	\td(s):=d(s)(1+s^2)
	=c_1+c_2s+c_3s^2-(s-k)e^{h(s\wedge w)}.  \notag
\end{gather}
Next, observe that the restriction of the function $\td$ to the interval $[w,\infty)$ is strictly convex (as 
a quadratic polynomial with the leading coefficient $c_3>0$) and hence 
\begin{equation*}
	\card\big(B(c_1,c_2,c_3)\cap[w,\infty)\big)\le1. 
\end{equation*}
It remains to show that 
\begin{equation*}
	\card\big(B(c_1,c_2,c_3)\cap(-\infty,w)\big)\le1. 
\end{equation*}   
Assume the contrary, so that there exist real numbers $u$ and $v$ in $B(c_1,c_2,c_3)$ such that $u<v<w$. 
Since the function $\td$ is differentiable on $(-\infty,w)$, it follows that 
\begin{equation*}
	\td(u)=\td'(u)=\td(v)=\td'(v)=0. 
\end{equation*}
The latter equalities constitute a system of four linear equations in $c_1,c_2,c_3,k$. Solving it, one finds that, in particular,  
\begin{equation*}
	c_3=\frac{h e^{hu}}{\de}\,\frac{\num}{\den},   
\end{equation*}
where $\num:=\de^2 - 4 \sinh^2\frac\de2$, $\den:=\de-2+(\de+2)e^{-\de}$, and $\de:=h (v - u)>0$.   
It is easy to see that $\num<0<\den$ for any $\de\in(0,\infty)$, which implies $c_3<0$ and thus contradicts \eqref{eq:c3>0}. 
Now Lemma~\ref{lem:contact} is completely proved, and thus so is Theorem~\ref{th:k}.   
\end{proof}
\end{proof} 

\begin{proof}[Proof of Proposition~\ref{prop:def}]
For $\vp\in(0,\infty)$, introduce 
\begin{equation*}
	\de(\vp):=\de_h(\vp):=\frac{u_*(h,\vp)-\vp}{\vp^2 e^{2(1 + \vp)h}}\,\big(e^{(1+\vp)h} (1+\vp h)-1\big);  
\end{equation*}
then 
$u_*(h,\vp)-\vp\eqs\de(\vp)$, 
$\de'(\vp)\eqs e^{(1+\vp)h} (1 + \vp h + \vp^2 (1 + \vp)h^2 )-(1+2 \vp h)
> (1 + \vp h)^2-(1+2\vp h)>0$, so that the continuous function $\de$ is (strictly) increasing, from $\de(0+)=-\infty$ to $\de(\infty-)=1$. Now part (i) of Proposition~\ref{prop:def} follows. 

It also follows that $u_*(h,\vp)-\vp\eqs\vp-\si_h^2$ for all $\vp\in(0,\infty)$, so that $u_*(h,\si^2)>\si^2$ --- assuming 
the condition $\si>\si_h$ of part (ii) of Proposition~\ref{prop:def}. 
Since $u_*(h,\vp)$ is continuous in $\vp\in(0,\infty)$, with $u_*(h,0+)=0$, to complete the proof of part (ii) of the proposition it remains to note that $u_*(h,\vp)$ is increasing in $\vp\in(0,\infty)$, which follows because, with $z:=\vp h$,  
\begin{align*}
	\pd{u_*(h,\vp)}\vp\, \frac{\big(1+\vp h-e^{-(1+\vp)h}\big)^2}{(1-e^{-(1+\vp)h})\, \vp}
	=&(z^2+2 z+2)e^{(1+1/\vp)z}-z^2-4 z-2 \\
  >& (z^2+2 z+2)(1+z)-z^2-4 z-2>0.  
\end{align*}

Consider next part (iii) of Proposition~\ref{prop:def}. Take here indeed any $w\in\{-1,0\}$. 
Then 
\begin{align*}
	r'_{w,1}
	(\vp)&\eqs r_{w,2}
	(\vp):=
	1-\frac{2 \varepsilon  e^{(\varepsilon+w)h} (1+\varepsilon h)}{\varepsilon^3 h^2+
   \left(h^2 \sigma ^2+2\right)\varepsilon +2 h \sigma ^2+4 \varepsilon^2 h}, \\ 
	r'_{w,2}
	(\vp)&\eqs 
-\varepsilon ^3 h^3 \left(\sigma ^2+\varepsilon ^2\right)-4 \varepsilon ^2 h^2 \left(\sigma
   ^2+\varepsilon ^2\right)-2 \left(2 \varepsilon ^3+3 \sigma ^2 \varepsilon \right)h -2
   \sigma ^2, 
\end{align*}
which is manifestly negative for all $\vp\in(|w|,\infty)$.  
So, $r_{w,2}
(\vp)$ decreases from $r_{w,2}
(|w|+)=\frac{h \left(\sigma ^2+|w|^2\right)(h |w|+2) }{2
   h \sigma ^2+h^2 |w|^3+|w| \left(h^2 \sigma ^2+2\right)+4 h |w|^2}>0$ to $r_{w,2}
(\infty-)=-\infty<0$ as $\vp$ increases from $|w|$ to $\infty$. 
Therefore, $r_{w,1}
(\vp)$ switches (just once) from increase to decrease as $\vp$ increases from $|w|$ to $\infty$. 
Since $r_{w,1}
(|w|+)=h |w| (w^2 + \si^2)\ge0
$ and $r_{w,1}
(\infty-)=-\infty<0$, one sees that $r_{w,1}(\vp)
$ switches (just once) in sign from $+$ to $-$ as $\vp$ increases from $|w|$ to $\infty$. 
This verifies part (iii) of the proposition.  
Now Proposition~\ref{prop:def} is completely proved. 
\end{proof}

\begin{proof}[Proof of Theorem~\ref{th:1}]
Take indeed any $w\in\{-1,0,1\}$. 
Introduce 
\begin{align}
\Pi_w&:=\{(\vp,\si)\in(0,\infty)^2\colon -\vp<w\le\si^2/\vp\}, \notag \\ 
E_w&:=\{\vp\in(0,\infty)\colon(\vp,\si)\in\Pi_w\text{ for some }\si\in(0,\infty)\} \notag \\
&=\big(0\vee(-w),\infty\big), \notag \\
	E_{w,\si}&:=\{\vp\in(0,\infty)\colon(\vp,\si)\in\Pi_w\}  
	=\left\{
	\begin{alignedat}{2}
	&(0,\si^2] &\text { if } &w=1, \\
	&(-w,\infty) &\text { if } &w\in\{-1,0\}  
	\end{alignedat}
	\right. \notag 
	\\ 
\intertext{for all $\si\in(0,\infty)$,}	
	\Si_{w,\vp}&:=\{\si\in(0,\infty)\colon(\vp,\si)\in\Pi_w\} 
	=\left\{
	\begin{alignedat}{2}
	&[\sqrt\vp,\infty) &\text { if } &w=1, \\
	&(0,\infty) &\text { if } &w\in\{-1,0\}   
	\end{alignedat}
	\right. \notag 
\intertext{ 
for all $\vp\in E_w$, }	
	m_w(\vp)&:=m_{w,\si}(\vp):=m_{h,w,\si}(\vp):=\MM_{h,w}P_{\vp,\,\si^2/\vp}
	=\vp\si^2\,\frac{e^{hw}-e^{-\vp h}}{\vp^2e^{hw}+\si^2e^{-\vp h}} \label{eq:m=}  	 
\end{align}
for all $(\vp,\si)\in\Pi_w$.  
By Theorem~\ref{th:sup=sup}, 
\begin{align*}
	\SS_{h,w,\si}=\max_{\vp\in E_{w,\si}}m_{h,w,\si}(\vp).  
\end{align*}
So, \eqref{eq:A=} and the equality in \eqref{eq:SS=} will follow once it is shown that 
\begin{equation}\label{eq:argmax=?}
	\argmax\limits_{\vp\in E_{w,\si}}m_{h,w,\si}(\vp)\overset{\text{(?)}}=\{\vp_{h,w,\si}\}. 
\end{equation}

Consider first the case $w=1$. 
Recalling the definition \eqref{eq:uu} of $u_*(h,\vp)$ and using Proposition~\ref{prop:def}, one has 
\begin{equation}\label{eq:r'(vp)=}
	m'_1(\vp)\eqs\si^2-u_*(h,\vp)\eqs\tvp_{h,\si}-\vp 
\end{equation}
for all $\vp\in E_{1,\si}=(0,\si^2]$, where $m'_1$ stands for the left-hand side derivative of $m_1$.  
So, $m_1(\vp)$ is increasing in $\vp\in(0,\tvp_{h,\si}]$ and decreasing in $\vp\in[\tvp_{h,\si},\infty)\cap(0,\si^2]$. 
Also, it was shown in the proof of Proposition~\ref{prop:def} that (i) $u_*(h,\vp)$ is increasing (from $0$) in $\vp\in(0,\infty)$ and (ii) $u_*(h,\vp)-\vp\eqs\vp-\si_h^2$ for all $\vp\in(0,\infty)$. 
So, $\si^2-u_*(h,\vp)$ is decreasing in $\vp\in(0,\si^2]$ from $\si^2>0$ to $\si^2-u_*(h,\si^2)\eqs\si_h^2-\si^2\eqs\si_h-\si$. 
Hence, by \eqref{eq:r'(vp)=}, 
in the case $\si\le\si_h$ one has $m_1'(\vp)>0$ for all $\vp\in(0,\si^2]$, so that the maximum of $m_1(\vp)$ in $\vp\in(0,\si^2]$ is attained only at $\vp=\si^2=\vp_{h,1,\si}$, by \eqref{eq:vp=}. 
Similarly, in the case $\si>\si_h$ the sign of $m_1'(\vp)$ changes only at the point $\tvp_{h,\si}$, from $+$ to $-$, as $\vp$ increases from $0$ to $\si^2$, so that the maximum of $m_1(\vp)$ in $\vp\in(0,\si^2]$ is in this case attained only at $\vp=\tvp_{h,\si}=\vp_{h,1,\si}$. 
This verifies \eqref{eq:argmax=?} 
for $w=1$.  

The case $w\in\{-1,0\}$ is simpler. Indeed, then 
$m'_w(\vp)\eqs r_{w,1}(\vp)$ for all $\vp\in E_{w,\si}$, where $r_{w,1}(\vp)$ is as in \eqref{eq:vp_-1}.  
Also, as shown at the end of the proof of Proposition~\ref{prop:def},  
$r_{w,1}(\vp)$ switches (just once, at $\vp=\vp_{h,w,\si}$) in sign from $+$ to $-$ as $\vp$ increases 
from $|w|$ to $\infty$,  
for each $w\in\{-1,0\}$\big).  
So, one has \eqref{eq:argmax=?} 
for $w\in\{-1,0\}$ as well. 
This proves \eqref{eq:A=} and the equality \eqref{eq:SS=} for all $w\in\{-1,0,1\}$. 

It remains to show that the inequality in \eqref{eq:SS=} holds and the constant factor $K_w(h)$ therein is the best possible. 
Introduce
\begin{equation*}
	r_{h,w,\si}(\vp):=\frac{m_{h,w,\si}(\vp)}{\si^2}
	=\vp\,\frac{e^{hw}-e^{-\vp h}}{\vp^2e^{hw}+\si^2e^{-\vp h}}   
\end{equation*}
for all $(\vp,\si)\in\Pi_w$, 
by \eqref{eq:m=}.  
Next, observe that $r_{h,w,\si}(\vp)$ strictly decreases in $\si\in\Si_{w,\vp}$;  
here and in what follows, it is assumed by default that $\vp\in E_w$. 
So, 
\begin{equation}\label{eq:rho=}
\begin{aligned}
		\rho_{h,w}(\vp)&:=\sup\Big\{\frac{m_{h,w,\si}(\vp)}{\si^2}\colon\si\in\Si_{w,\vp}\Big\} \\
		&=\sup\{r_{h,w,\si}(\vp)\colon\si\in\Si_{w,\vp}\} 
	=r_{h,w,\si_w(\vp)}(\vp),\quad\text{where} \\ 
\si_w(\vp)&:=\inf\Si_{w,\vp}=(w\vee0)\sqrt\vp. 	
\end{aligned}	
\end{equation}

Consider now the case $w=1$, so that $E_w=E_1=(0,\infty)$, $\si_w(\vp)=\si_1(\vp)=\sqrt\vp$, and 
\begin{equation}\label{eq:rho=rho=}
	\rho_{h,w}(\vp)=\rho_{h,1}(\vp)=r_{h,1,\sqrt\vp}(\vp)=\frac{e^{(1+\vp)h}-1}{1+\vp e^{(1+\vp)h}}. 
\end{equation}
Now note that $\rho_{h,1}'(\vp)\eqs1+(1+\vp)h-e^{(1+\vp)h}<0$. 
Hence, 
$\rho_{h,1}(\vp)<\rho_{h,1}(0+)=e^h-1=K_1(h)$, 
which, together with \eqref{eq:m=} and \eqref{eq:rho=}, shows that the inequality in \eqref{eq:SS=} holds and that the constant factor $K_1(h)$ there is the best possible --- in the case $w=1$. 

Next, consider the case $w\in\{-1,0\}$, when $E_w=(-w,\infty)$, $\si_w(\vp)=0$, and 
\begin{equation}\label{eq:r<}
	r_{h,w,\si}(\vp)<r_{h,w,0+}(\vp)=\rho_{h,w}(\vp)=\frac{1-e^{-(w+\vp)h}}\vp  
\end{equation}
for all $\si\in\Si_{w,\vp}=(0,\infty)$. 
On the other hand, with 
\begin{equation*}
	u:=-(1+\vp h), 
\end{equation*}
one has 
\begin{align*}
	\rho'_{h,w}(\vp) \eqs (1+\vp h)e^{-(w+\vp)h}-1
	\eqs -ue^u-e^{h w-1} 
	&\eqs u-L_{-1}(-e^{h w-1}) \\
	&\eqs \vp_{h,w,*}-\vp, 
\end{align*}
where 
\begin{equation*}
	\vp_{h,w,*}:=-\frac{1+L_{-1}(-e^{h w-1})}h\in[-w,\infty). 
\end{equation*}
So,  
\begin{equation*}
	\rho_{h,w}(\vp)\le\rho_{h,w}(\vp_{h,w,*}+)=\frac h{-L_{-1}(-e^{h w-1})}=K_w(h). 
\end{equation*}
This, 
together with \eqref{eq:m=}, \eqref{eq:rho=}, and \eqref{eq:r<}, shows that, in the case $w\in\{-1,0\}$ as well, 
the inequality in \eqref{eq:SS=} holds and that the constant factor $K_w(h)$ there is the best possible. 
This completes the proof of Theorem~\ref{th:1}. 
\end{proof}

\begin{proof}[Proof of Proposition~\ref{prop:mono}] 
The main idea of the proof is to use positive association of r.v.'s; see e.g.\ \cite{lehmann66,esary}. 
Take any r.v.\ $X$. 
For any $h\in(0,\infty)$ and $w\in\R$, let $\tX=\tX_{h,w}$ and $\tY=\tY_{h,w}$ be any r.v.'s such that 
\begin{equation*}
	\E f(\tX)g(\tY)=\frac{\E f(X)e^{h(X\wedge w)}}{\E e^{h(X\wedge w)}}\,\frac{\E g(X)e^{h(X\wedge w)}}{\E e^{h(X\wedge w)}}
\end{equation*}
for all nonnegative Borel functions $f$ and $g$ on $\R$; cf.\ \eqref{eq:E_h}. 
It should be clear that such r.v.'s $\tX$ and $\tY$ do exist; moreover, necessarily they are independent copies of each other, and also $\E\tX=\E\tY=\E_{h,w}X$. 

Letting now $g_1(x):=g_{w,1}(x):=x\wedge w$, one has 
\begin{equation}\label{eq:pd>0}
	\pd{}h\E_{h,w}X=\E\tX g_1(\tX)-\E\tX\E g_1(\tX)
	=\tfrac12\E(\tX-\tY)\big(g_1(\tX)-g_1(\tY)\big)\ge0, 
\end{equation}
because the function $g_1$ is nondecreasing and hence $(\tX-\tY)\big(g_1(\tX)-g_1(\tY)\big)\ge0$. 
This shows that $\E_{h,w}X$ is nondecreasing in $h$. Similarly, using (say) the right-hand side partial derivatives $\pd{}w$ in $w$, with $g_2(x):=\pd{(x\wedge w)}w=\ii{x>w}$ in place of $g_1(x)$, one verifies that $\E_{h,w}X$ is nondecreasing in $w$. 
Thus, part (I) of Proposition~\ref{prop:mono} is proved. 

To prove part (II), take any 
$w\in(i_X,\infty)$, and then take any $c\in(i_X,w\wedge s_X)$. Then on the event $C:=\{\tX<c<\tY\}=\{\tX_{h,w}<c<\tY_{h,w}\}$ one has \break 
$g_1(\tX)<c<g_1(\tY)$ and hence $(\tX-\tY)\big(g_1(\tX)-g_1(\tY)\big)>0$; also, $\P(C)=\P(\tX<c) \P(c<\tY)>0$, which implies that the inequality in \eqref{eq:pd>0} is strict. 

Part (III) is proved similarly \big(here it is enough to prove that $\pd{}w\E_{h,w}X>0$ for all $w\in(i_X,s_X)$\big). 

To prove part (IV), observe that, by Theorems~\ref{th:sup=sup} and \ref{th:1}, 
\begin{equation*}
	\SS_{h,w,\si}=\MM_{h,w}P_{u(w),v(w)},
\end{equation*}
where, recall, $P_{u,v}$ is the zero-mean distribution on the set $\{-u,v\}$, as defined before the statement of Theorem~\ref{th:1}; $u(w)=u_{h,\si}(w)$ and $v(w)=v_{h,\si}(w)$ are positive real numbers depending only on $h,w,\si$ and such that $-u(w)<w\le v(w)$; moreover, $w<v(w)$ unless $w\ge\si/\si_h$ \big(recall Remark~\ref{rem:scaling}, according to which the condition $\si\le\si_h$ in \eqref{eq:vp=} for $w=1$ should be transformed into $\frac\si w\le\si_h$ for a general $w\in(0,\infty)$\big). 
Therefore and in view of part (III) of Proposition~\ref{prop:mono}, the first inequality in  
\begin{equation}\label{eq:w,w_1}
	\SS_{h,w,\si}=\MM_{h,w}P_{u(w),v(w)}
	\le\MM_{h,w_1}P_{u(w),v(w)}\le\SS_{h,w_1,\si}
\end{equation}
is strict 
for any $w_1\in\big(w,v(w)\big)$, and such a point $w_1$ exists unless $w\ge\si/\si_h$. 
If now $w\ge\si/\si_h$, then, again by \eqref{eq:vp=} and Remark~\ref{rem:scaling}, for any $w_1\in(w,\infty)$ one has $w_1\ge\si/\si_h$ and hence 
$u(w_1)=w_1\vp_{h,1,\si/w_1}=w_1(\si/w_1)^2=\si^2/w_1$ and $v(w_1)=\si^2/u(w_1)=w_1\ne w=v(w)$, whence $P_{u(w),v(w)}\ne P_{u(w_1),v(w_1)}$.  
Hence, by Remark~\ref{rem:unique}, the second inequality in \eqref{eq:w,w_1} is strict. 
So, whether or not the condition $w_1\ge\si/\si_h$ holds, for any $w\in\R$ and all $w_1$ in a right neighborhood of $w$, one has $\SS_{h,w,\si}<\SS_{h,w_1,\si}$. 
Thus, $\SS_{h,w,\si}$ is increasing in $w\in\R$. 
That $\SS_{h,w,\si}$ is increasing in $h\in(0,\infty)$
can be shown similarly, and with less difficulty, because in this setting $w$ ``is not moving'', and so, by part (II) of Proposition~\ref{prop:mono}, the second inequality in the formula corresponding to \eqref{eq:w,w_1} will always be strict. 
Proposition~\ref{prop:mono} is now completely proved. 
\end{proof}


\bibliographystyle{acm}
\bibliography{C:/Users/Iosif/Documents/mtu_home01-30-10/bib_files/citations}

\def\cprime{$'$} \def\polhk#1{\setbox0=\hbox{#1}{\ooalign{\hidewidth
  \lower1.5ex\hbox{`}\hidewidth\crcr\unhbox0}}}
  \def\polhk#1{\setbox0=\hbox{#1}{\ooalign{\hidewidth
  \lower1.5ex\hbox{`}\hidewidth\crcr\unhbox0}}}
  \def\polhk#1{\setbox0=\hbox{#1}{\ooalign{\hidewidth
  \lower1.5ex\hbox{`}\hidewidth\crcr\unhbox0}}} \def\cprime{$'$}
  \def\polhk#1{\setbox0=\hbox{#1}{\ooalign{\hidewidth
  \lower1.5ex\hbox{`}\hidewidth\crcr\unhbox0}}}
\begin{thebibliography}{10}

\bibitem{knuth96}
{\sc Corless, R.~M., Gonnet, G.~H., Hare, D. E.~G., Jeffrey, D.~J., and Knuth,
  D.~E.}
\newblock On the {L}ambert {$W$} function.
\newblock {\em Adv. Comput. Math. 5}, 4 (1996), 329--359.

\bibitem{esary}
{\sc Esary, J.~D., Proschan, F., and Walkup, D.~W.}
\newblock Association of random variables, with applications.
\newblock {\em Ann. Math. Statist. 38\/} (1967), 1466--1474.

\bibitem{ky_fan}
{\sc Fan, K.}
\newblock Minimax theorems.
\newblock {\em Proc. Nat. Acad. Sci. U. S. A. 39\/} (1953), 42--47.

\bibitem{hoeff-extr}
{\sc Hoeffding, W.}
\newblock The extrema of the expected value of a function of independent random
  variables.
\newblock {\em Ann. Math. Statist. 26\/} (1955), 268--275.

\bibitem{karlin-studden}
{\sc Karlin, S., and Studden, W.~J.}
\newblock {\em Tchebycheff systems: {W}ith applications in analysis and
  statistics}.
\newblock Pure and Applied Mathematics, Vol. XV. Interscience Publishers John
  Wiley \& Sons, New York-London-Sydney, 1966.

\bibitem{karr}
{\sc Karr, A.~F.}
\newblock Extreme points of certain sets of probability measures, with
  applications.
\newblock {\em Math. Oper. Res. 8}, 1 (1983), 74--85.

\bibitem{kemper-dual}
{\sc Kemperman, J. H.~B.}
\newblock On the role of duality in the theory of moments.
\newblock In {\em Semi-infinite programming and applications ({A}ustin, {T}ex.,
  1981)}, vol.~215 of {\em Lecture Notes in Econom. and Math. Systems}.
  Springer, Berlin, 1983, pp.~63--92.

\bibitem{kneser}
{\sc Kneser, H.}
\newblock Sur un th\'eor\`eme fondamental de la th\'eorie des jeux.
\newblock {\em C. R. Acad. Sci. Paris 234\/} (1952), 2418--2420.

\bibitem{krein-nudelman}
{\sc Kre{\u\i}n, M.~G., and Nudel{\cprime}man, A.~A.}
\newblock {\em The {M}arkov moment problem and extremal problems}.
\newblock American Mathematical Society, Providence, R.I., 1977.
\newblock Ideas and problems of P. L. {\v{C}}eby{\v{s}}ev and A. A. Markov and
  their further development, Translated from the Russian by D. Louvish,
  Translations of Mathematical Monographs, Vol. 50.

\bibitem{lehmann66}
{\sc Lehmann, E.~L.}
\newblock Some concepts of dependence.
\newblock {\em Ann. Math. Statist. 37\/} (1966), 1137--1153.

\bibitem{pin98}
{\sc Pinelis, I.}
\newblock Optimal tail comparison based on comparison of moments.
\newblock In {\em High dimensional probability ({O}berwolfach, 1996)}, vol.~43
  of {\em Progr. Probab.} Birkh\"auser, Basel, 1998, pp.~297--314.

\bibitem{winzor}
{\sc Pinelis, I.}
\newblock Exact lower bounds on the exponential moments of {W}insorized and
  truncated random variables.
\newblock {\em J. App. Probab. 48\/} (2011), 547--560.

\bibitem{nonlinear}
{\sc Pinelis, I., and Molzon, R.}
\newblock Berry-{E}ss{\'e}{e}n bounds for general nonlinear statistics, with
  applications to {P}earson's and non-central {S}tudent's and {H}otelling's
  (preprint), ar{X}iv:0906.0177v1 [math.{ST}].

\bibitem{sion}
{\sc Sion, M.}
\newblock On general minimax theorems.
\newblock {\em Pacific J. Math. 8\/} (1958), 171--176.

\bibitem{von_Neumann}
{\sc von Neumann, J., and Morgenstern, O.}
\newblock {\em Theory of {G}ames and {E}conomic {B}ehavior}.
\newblock Princeton University Press, Princeton, New Jersey, 1944.

\end{thebibliography}

\end{document}